\documentclass[9pt]{amsart}
\usepackage{latexsym, amssymb, amsmath, mathrsfs, amsfonts}

\newtheorem{thm}{Theorem}[section]

\theoremstyle{remark}
\newtheorem{rmk}[thm]{Remark}
\theoremstyle{definition}

\title{The optimal diameter estimate for positive Yamabe metrics}

\author{Kazuo Akutagawa${}^*$}
\email{akutagawa@math.chuo-u.ac.jp}
\address{Department of Mathematics, Chuo University,
Tokyo, Japan}

\author{Mijia Lai}
\email{laimijia@sjtu.edu.cn}
\address{School of Mathematical Science, Shanghai Jiao Tong University, Shanghai 200240, China}

\author{Chao Li${}^{**}$}
\email{chaoli@nyu.edu}
\address{Courant Institute of Mathematical Sciences, New York University, USA}

\author{Harish Seshadri}
\email{harish@iisc.ac.in}
\address{Department of Mathematics, Institute of Science, Bangalore, India}

\author{Zhixin Wang${}^{***}$}
\email{jhin@sjtu.edu.cn}
\address{School of Mathematical Science, Shanghai Jiao Tong University, Shanghai 200240, China}

\thanks{${}^*$\
supported in part by the Grants-in-Aid for Scientific Research (C),
Japan Society for the Promotion of Science, No.~24K06718.}

\thanks{${}^{**}$\
partially supported by NSF grant DMS-2303624 and a Sloan Fellowship.}

\thanks{${}^{***}$\ supported in part by the China Postdoctoral Science Foundation No. 2025M773102.}

\date{September, 2026
}
\begin{document}
\maketitle
\markboth{optimal diameter estimate}
{Kazuo Akutagawa, Mijia Lai, Chao Li, Harish Seshadri, Zhixin Wang}

\begin{abstract}
In this paper, we prove the optimal diameter estimate for positive Yamabe metrics on a connected, closed manifold. It is a scalar curvature version of both classical Myers' diameter estimate and Cheng's maximal diameter theorem on Ricci curvature.
\end{abstract}
\maketitle

\section{Introduction and main result}

Ricci curvature reveals much on the global geometry and topology of manifolds (cf.\,\cite{Gromov-Book, GHL}).
Two of such results on Ricci curvature are classical Myers' diameter estimate and Cheng's maximal diameter theorem below.

\begin{thm}[Myers\,\cite{M}, Cheng\,\cite{Cg}]
Let $g$ be a Riemannian metric on a connected, complete $n$-manifold $N\,(n \geq 3)$.
Assume that the Ricci curvature ${\rm Ric}_g$ of $g$ satisfies \\
${\rm Ric}_g \geq (n- 1)k^2\cdot g$\ $($ for some $k > 0$ $)$.
Then, the following statements hold $:$ \\
$({\rm i})$\ \ ${\rm diam}(N, g) \leq \frac{\pi}{k}$, \\
$({\rm ii})$\ \ $N$ is compact,\\
$({\rm iii})$\ \ $|\pi_1(N)| < \infty$. \\
Moreover, if the equality ${\rm diam}(N, g) = \frac{\pi}{k}$ holds, then $(N, g)$ is isometric to $S^n(\frac{1}{k})$,
where $S^n(\frac{1}{k})$ denotes the round $n$-sphere of  radius $\frac{1}{k} > 0$ in $\mathbb{R}^{n+1}$.
\end{thm}

Clearly, one cannot relax the Ricci lower bound to a lower bound on scalar curvature,
as the following typical example shows :
Let $g_{r, s}$ denote the standard product metric on $S^{n-\ell}(r) \times S^{\ell}(s)\ (n \geq 3, \ell \geq 1\ {\rm with}\ n - \ell \geq 2)$.
Then, if $r \leq \frac{1}{k}\sqrt{\frac{(n-\ell)(n-\ell-1)}{n(n-1)}}$, 
then $R_{g_{r, s}} = \frac{(n-\ell)(n-\ell-1)}{r^2} + \frac{\ell(\ell-1)}{s^2} \geq n(n-1)k^2$.
But, ${\rm diam}(S^{n-\ell}(r) \times S^{\ell}(s), g_{r,s}) \geq \pi s \rightarrow \infty$\ as\ $s \rightarrow \infty$.

Yamabe metrics of positive scalar curvature ({\it positive Yamabe metrics} for brevity) 
enjoy two basic properties parallel to classical consequences of a metric with Ricci curvature bounded below by a constant. 
If $g$ is a Yamabe metric
with $R_g>0$, then
\begin{equation}\notag
 \lambda_1(-\Delta_g)\geq \frac{R_g}{n-1},
\end{equation}
and
\begin{equation}\notag
 R_g\operatorname{Vol}(M,g)^{2/n}
 \leq n(n-1)\operatorname{Vol}(S^n, g_{{}_{S^n}})^{2/n}.
\end{equation}
Here $\lambda_1(-\Delta_g)$ denotes the first nonzero eigenvalue and
$g_{{}_{S^n}}$ is the unit round metric. The first inequality is the
Yamabe analogue of the Lichnerowicz estimate, while the second is the
analogue of Bishop's volume comparison; see
\cite[Proposition~2.1]{A2}.

In view of these parallels, it is natural to seek a Myers type
diameter estimate for positive Yamabe metrics. In fact, the first-named author had previously obtained the following 
(see the figure in \cite[page 241]{Ak} for a typical example when $\ell = 1$) :

\begin{thm}$($\cite[Proposition\,2.2]{Ak}$)$
Let $\overset{\lor}{g}$ be a positive Yamabe metric on a connected, closed $n$-manifold $M\,(n \geq 3)$.
Assume that $R_{\overset{\lor}{g}} \geq n(n- 1)k^2$\ $($ for some $k > 0$ $)$.
Then, there exists a positive constant $C_n > 0$ depending only on $n$ such that
$$
{\rm diam}(M, \overset{\lor}{g}) \leq \frac{C_n}{k}.
$$
\end{thm}

Our main result of this paper is the following diameter estimate, which is the optimal improvement of the above:

\begin{thm} \label{thm:main}
Let ${\overset{\lor}{g}}$ be a positive Yamabe metric on a connected, closed $n$-manifold $M\,(n \geq 3)$.
Assume that $R_{\overset{\lor}{g}} \geq n(n - 1)k^2$\ $($ for some $k > 0$ $)$.
Then,
$$
{\rm diam}(M, {\overset{\lor}{g}}) \leq \frac{\pi}{k}.
$$
Moreover, if ${\rm diam}(M, {\overset{\lor}{g}}) = \frac{\pi}{k}$,
then $(M, {\overset{\lor}{g}})$ is isometric to the round $n$-sphere $S^n(\frac{1}{k})$.
\end{thm}

It should be remarked the following.
For each conformal class $C$ on a given closed manifold,
a Yamabe metric in $C$ always exists and has constant scalar curvature.
Hence, it is a suitable representative of $C$.

The remaining sections are organized as follows.
Section\,$2$ contains some necessary definitions and preliminary geometric results on Yamabe constants and Yamabe metrics.
It then is devoted to the proof of Theorem\,1.3.
In Sections\,$3$ and $4$, we will give similar diameter estimates
for relative Yamabe metrics and almost conical Yamabe metrics.

\section{Preliminaries and the proof of main result}

We first review the definitions of Yamabe constants and Yamabe metrics.
Let $M^n$ be a connected, closed $n$-manifold ($n \geq 3$).
It is well known that a Riemannian metric on $M$ is {\it Einstein} if and only if it is a critical point of
the {\it normalized Einstein-Hilbert functional} $E$ on the space $\mathcal{M}(M)$ of all Riemannian metrics on $M$
$$
E : \mathcal{M}(M) \rightarrow \mathbb{R},\quad g \mapsto E(g) := \frac{\int_MR_gd\mu_g}{{\rm Vol}(M, g)^{(n-2)/n}},
$$
where $R_g, d\mu_g$ and ${\rm Vol}(M, g)$ denote respectively the scalar curvature of $g$, the volume measure of $g$
and the volume of $(M, g)$.
The restriction $E|_{[g]}$ of $E$ to any conformal class $[g] := \{ e^{2f}\,g\ |\ f \in C^{\infty}(M) \}$ is always bounded from below.
Hence, we can consider the following conformal invariant
$$
Y(M, [g]) := \inf_{\widetilde{g} \in [g]}E(\widetilde{g}),
$$
which is called the {\it Yamabe constant} of $(M, [g])$.
A remarkable theorem of Yamabe, Trudinger, Aubin and Schoen asserts that
each conformal class $[g]$ always contains a metric $\overset{\lor}{g}$, called a {\it Yambe metric},
which realizes the minimum (cf.\,\cite{LP}, \cite{Sc-1}, \cite{Au-Book})
$$
Y(M, [g]) = E(\overset{\lor}{g}).
$$
This metric must have constant scalar curvature
$$
R_{\overset{\lor}{g}} = Y(M, [g])\cdot V_{\overset{\lor}{g}}^{-2/n},
$$
where $V_{\overset{\lor}{g}} = {\rm Vol}(M, \overset{\lor}{g})$.
Aubin\,\cite{Au} proved that
$$
Y(M^n, C) \leq Y(S^n, [g_{{}_{S^n}}]) = n(n-1) V_{g_{{}_{S^n}}}^{2/n}
$$
for any conformal class $C$ on $M$. 

The Yamabe constant $Y(M, [g])$ is rewritten as
$$
Y(M, [g]) = \inf_{\substack{u \in W^{1, 2}(M) \\ u \not\equiv 0}}\frac{\alpha_n\int_M|du|^2d\mu_g + \int_MR_gu^2d\mu_g}
{\Big{(}\int_M|u|^{2n/(n-2)}d\mu_g\Big{)}^{(n-2)/n}},
$$
where $W^{1, 2}(M)$ denotes the Sobolev space of functions with $L^2$ first weak derivatives and \\
$\alpha_n := \frac{4(n - 1)}{n - 2} > 0$.
If $Y(M, [g]) > 0$, the above and $R_{\overset{\lor}{g}} = Y(M, [g])\cdot V_{\overset{\lor}{g}}^{-2/n}$ imply
$$
\Big{(}\int_M|u|^{2n/(n-2)}d\mu_{\overset{\lor}{g}}\Big{)}^{(n-2)/n} \leq
\frac{\alpha_n}{Y(M, [g])}\int_M|du|^2d\mu_{\overset{\lor}{g}}
+ \frac{1}{V_{\overset{\lor}{g}}^{2/n}}\int_Mu^2d\mu_{\overset{\lor}{g}}\quad {\rm for}\ \ \forall u \in W^{1,2}(M).
$$
Moreover, mlutiplying the above by $\frac{1}{V_{\overset{\lor}{g}}^{(n-2)/n}}$, we also get
\begin{equation}
\Big{(}\int_M|u|^{2n/(n-2)}\frac{1}{V_{\overset{\lor}{g}}}d\mu_{\overset{\lor}{g}}\Big{)}^{(n-2)/n} \leq
\frac{\alpha_n}{R_{\overset{\lor}{g}}}\int_M|du|^2\frac{1}{V_{\overset{\lor}{g}}}d\mu_{\overset{\lor}{g}}
+ \int_Mu^2\frac{1}{V_{\overset{\lor}{g}}}d\mu_{\overset{\lor}{g}}
\end{equation}
for any $u \in W^{1,2}(M)$.
Here, we notice that $\frac{1}{V_{\overset{\lor}{g}}}d\mu_{\overset{\lor}{g}}$ is a probability measure on $M$,
that is, $\frac{1}{V_{\overset{\lor}{g}}}\mu_{\overset{\lor}{g}}(M) = 1$.
The diameter ${\rm diam}(M, {\overset{\lor}{g}})$ of $(M, {\overset{\lor}{g}})$ is also rewritten as
$$
{\rm diam}(M, {\overset{\lor}{g}}) = \sup\{ ||\widetilde{f}||_{L^{\infty}}\ \big{|}\ f \in W^{1, 2}(M),\ ||df||_{L^{\infty}} \leq 1 \},
$$
where $\widetilde{f}(x, y) := f(x) - f(y)$ for $x, y \in M$.
Remark that, since $(M, g)$ is connected and closed, there always exists $f \in W^{1,2}(M)$
such that $||df||_{\infty} \leq 1$ and $||\widetilde{f}||_{\infty} = {\rm diam}(M, g)$. \\

By replacing $\overset{\lor}{g}$ by $k^2\overset{\lor}{g}$,
it is enough to prove Theorem\,1.3 when $k=1$.
For the proof of the diameter estimate in Theorem\,1.3, we first prepare following result:

\begin{thm}$(${\rm Bakry--Ledoux}\,\cite[Theorem\,2]{BL}$).$
Let $(M, g)$ be a connected, closed Riemannian $n$-manifold $(n \geq 3)$.
Assume that, for some $A > 0$,
\begin{equation}
||f||^2_{2n/(n-2)} \leq A||df||^2_2 + ||f||^2_2\quad {\rm for}\ \ f \in W^{1,2}(M).
\end{equation}
Then,
$$
{\rm diam}(M, g) \leq \pi\sqrt{\frac{n(n-2)A}{4}}.
$$
Here all the $L^q$-norms below are taken with respect to the
probability measure $V_g^{-1}d\mu_g$. 
\end{thm}

\noindent\begin{proof}[Proof of ${\rm diam}(M, \overset{\lor}{g}) \leq \pi$.]
From the inequality\,$(1)$, we can take $A = \frac{\alpha_n}{R_{\overset{\lor}{g}}} > 0$ in the inequality\,$(2)$ of Theorem\,2.1.
It immediately follows from $R_{\overset{\lor}{g}} = Y(M, [g])\cdot V_{\overset{\lor}{g}}^{-2/n}$ and $R_{\overset{\lor}{g}} \geq n(n-1)$ that
$$
{\rm diam}(M, {\overset{\lor}{g}}) \leq \pi\sqrt{\frac{n(n-1)}{R_{\overset{\lor}{g}}}} \leq \pi.
$$
\end{proof}

For the proof of the equality case in Theorem\,1.3, we also prepare the following result:

\begin{thm}$(${\rm Bakry--Ledoux}\,\cite[Theorem\,4]{BL}$).$
Let $(M, g)$ be a connected, closed Riemannian $n$-manifold $(n \geq 3)$ satisfying the Sobolev inequality $(2)$.
Assume that there exists a function $f \in W^{1,2}(M)$ such that $||df||_{\infty} \leq 1$ and $||\widetilde{f}||_{\infty} = \pi$.
Then, there exists a nonconstant function $u \in W^{1,2}(M)$ such that the equality holds in $(2)$ as below $:$
\begin{equation}
||u||^2_{2n/(n-2)} = A||du||^2_2 + ||u||^2_2.
\end{equation}
Moreover, if we set $X := \sin(f)$, for every $\varphi \in C^{\infty}(\mathbb{R})$,
$$
\Delta_g\varphi(X) = (1 - X^2)\varphi''(X) - nX\varphi'(X).
$$
\end{thm}

\begin{proof}[Proof of the equality case $: {\rm diam}(M, {\overset{\lor}{g}}) = \pi$.]
The estimate already proved gives
\[
 \pi={\rm diam}(M, \overset{\lor}{g})
 \leq \pi\sqrt{\frac{n(n-1)}{R_{\overset{\lor}{g}}}}
 \leq \pi .
\]
Consequently,
\[
 R_{\overset{\lor}{g}}=n(n-1),
 \qquad
 \frac{\alpha_n}{R_{\overset{\lor}{g}}}=\frac{4}{n(n-2)}.
\]
Theorem\,2.2 therefore gives a nonconstant function
$X\colon M\longrightarrow[-1,1]$ such that, for every
$\varphi\in C^\infty(\mathbb R)$,
\[
 \Delta_{\overset{\lor}{g}}\varphi(X)
 =(1-X^2)\varphi''(X)-nX\varphi'(X).
\]
Taking $\varphi(x)=x$ gives
\[
 \Delta_{\overset{\lor}{g}}X=-nX.
\]
Taking $\varphi(x)=x^2$ and comparing with
$\Delta_{\overset{\lor}{g}}(X^2)=2X\Delta_{\overset{\lor}{g}}X+2|dX|^2$ gives
\[
 |dX|^2=1-X^2.
\]
\underline{Claim\,1.}\ \ ${\rm Vol}(M, {\overset{\lor}{g}}) = {\rm Vol}(S^n(1)) =:\sigma_n$. \\
\quad\\
{\it Proof of Claim\,{\rm 1}} :\ We first record the critical point structure of $X$. Since $M$ is compact and
$X$ is nonconstant, $X$ has a maximum and a minimum. At every critical
point, the second equation above gives $X=\pm1$; hence
\[
 \max_M X=1,
 \qquad
 \min_M X=-1,
\]
and the critical points of $X$ are precisely the points of the two sets
\[
 P_+:=\{X=1\},
 \qquad
 P_-:=\{X=-1\}.
\]
Let $p\in P_+$ and let $H=({}^{\overset{\lor}{g}}\nabla^2X)_p$. In geodesic normal coordinates
centered at $p$,
\[
 X(y)=1+\frac12\langle Hy,y\rangle+O(|y|^3).
\]
Comparing the quadratic terms in $|dX|^2=1-X^2$ gives
$H^2=-H$. Moreover,
\[
 {\rm tr}_{\overset{\lor}{g}}\,H=\Delta_{\overset{\lor}{g}}X(p)=-n.
\]
Thus every eigenvalue of $H$ is either $0$ or $-1$, and their sum is $-n$;
therefore $H=-I$. Similarly, at every point of $P_-$ one obtains $H^2=H$
and ${\rm tr}_{\overset{\lor}{g}}\,H=n$, hence $H=I$. It follows that $X$ is a Morse function
whose critical points have only indices $0$ and $n$.

Since $M$ is connected, Morse theory now gives
\[
 \#P_- =1.
\]
Indeed, every minimum creates a connected component of a small sublevel set,
and in the absence of critical points of index $1$ no two such components can
be joined. Applying the same argument to $-X$ gives
\[
 \#P_+=1.
\]
Denote the unique points in $P_+$ and $P_-$ by $p_+$ and $p_-$,
respectively.

Set
\[
 t=\arccos X.
\]
On $M\setminus\{p_+,p_-\}$, the two equations for $X$ imply
\[
 |dt|=1,
 \qquad
 \Delta_{\overset{\lor}{g}}t=(n-1)\cot t.
\]
For $0<\tau<\pi$, put
\[
 \Sigma_\tau:=\{t=\tau\},
 \qquad
 A(\tau):={\rm Area}(\Sigma_\tau).
\]
If $0<a<b<\pi$, the divergence theorem and the coarea formula give
\begin{align*}
 A(b)-A(a)
 &=\int_{\{a<t<b\}}\Delta_{\overset{\lor}{g}}t\,d\mu_{\overset{\lor}{g}}\\
 &=\int_a^b(n-1)\cot\tau\,A(\tau)\,d\tau.
\end{align*}
Hence
\[
 A'(\tau)=(n-1)\cot\tau\,A(\tau),
\]
and therefore
\[
 A(\tau)=C\sin^{n-1}\tau
\]
for some $C>0$.

Let $\sigma_{n-1}$ denote the volume of the unit round $(n-1)$-sphere. At $p_+$ we
have ${}^{\overset{\lor}{g}}\nabla^2X=-I$, and hence, with $r={\rm dist}_{\overset{\lor}{g}}(p_+,\cdot)$,
\[
 X=1-\frac12r^2+O(r^3),
 \qquad
 t=r+O(r^2).
\]
It follows that
\[
 A(\tau)=\sigma_{n-1}\tau^{n-1}+o(\tau^{n-1})
 \quad\text{as }\tau\downarrow0.
\]
Thus $C=\sigma_{n-1}$. The coarea formula now yields
\[
 V_{\overset{\lor}{g}}
 =\int_0^\pi A(\tau)\,d\tau
 =\sigma_{n-1}\int_0^\pi\sin^{n-1}\tau\,d\tau
 =\sigma_n.
\]
This completes the proof of Claim\,1.\\

Since $\overset{\lor}{g}$ is a Yamabe metric, then
\[
 Y(M,[g])
 =R_{\overset{\lor}{g}}V_{\overset{\lor}{g}}^{2/n}
 =n(n-1)\sigma_n^{2/n}
 =Y(S^n,[g_{{}_{S^n}}]).
\]
The equality case in Aubin's inequality therefore implies that
$(M,[g])$ is conformally diffeomorphic to the standard sphere
(cf.\,\cite{LP,Sc-1}). Let
$\Psi\colon M\longrightarrow S^n$ be such a conformal diffeomorphism and set
\[
 h=(\Psi^{-1})^*\overset{\lor}{g}.
\]
Then $h$ is conformal to $g_{{}_{S^n}}$ and has scalar curvature $n(n-1)$.
By Obata's classification \cite{Obata}, there are a constant $c>0$ and a
conformal diffeomorphism $\phi\colon S^n\longrightarrow S^n$ such that
\[
 h=c\,\phi^*g_{{}_{S^n}}.
\]
Taking scalar curvatures gives
$n(n-1)=c^{-1}n(n-1)$, so $c=1$. Consequently,
\[
 \overset{\lor}{g}
 =\Psi^*h
 =(\phi\circ\Psi)^*g_{{}_{S^n}}.
\]
Thus $\phi\circ\Psi\colon(M, \overset{\lor}{g})\longrightarrow
(S^n,g_{{}_{S^n}})$ is an isometry.
\end{proof}
\quad\\
\begin{rmk}\label{rem:complete-finite-volume}
In view of the Myers theorem, it is natural to ask
whether the assumption that $M$ is closed in Theorem~\ref{thm:main} can
be replaced by completeness. In the finite-volume globally
minimizing setting, the answer is affirmative.

For a possibly noncompact Riemannian $n$-manifold $(M,g)$,
$n\geq 3$, define
\[
Y(M,[g])
:=
\inf_{0\neq u\in C_c^\infty(M)}
\frac{\displaystyle
\int_M
\left(
\alpha_n |du|_g^2+R_g u^2
\right)\,d\mu_g}
{\displaystyle
\left(
\int_M |u|^{\frac{2n}{n-2}}\,d\mu_g
\right)^{\frac{n-2}{n}}}.
\]
Let $(M^n, \widetilde{g})\ (\widetilde{g} \in [g])$ be complete and connected, and assume that
\[
V_{\widetilde{g}}:=\operatorname{Vol}(M, \widetilde{g}) < \infty,
\qquad
R_{\widetilde{g}}>0,
\]
where $R_{\widetilde{g}}$ is constant. Suppose moreover that
$\widetilde{g}$ is globally Yamabe minimizing in the sense that
\[
Y(M,[g])
=
R_{\widetilde{g}}\cdot V_{\widetilde{g}}^{2/n}.
\]
Then
\[
\operatorname{diam}(M, \widetilde{g})
\leq
\pi\sqrt{\frac{n(n-1)}{R_{\widetilde{g}}}},
\]
and consequently $M$ is compact.

Indeed, set
\[
d\mu:=\frac{1}{V_{\widetilde{g}}}d\mu_{\widetilde{g}}.
\]
The globally minimizing property gives
\[
\|u\|_{L^{\frac{2n}{n-2}}(d\mu)}^2
\leq
\|u\|_{L^2(d\mu)}^2
+
\frac{\alpha_n}{R_{\widetilde{g}}}
\int_M |du|_{\widetilde{g}}^2\,d\mu
\]
for every $u\in C_c^\infty(M)$. By a standard cutoff argument,
using completeness and finite volume, this inequality extends
to the bounded Lipschitz functions needed in the Bakry--Ledoux theorem.
Their diameter estimate therefore yields the asserted bound.
Since $(M, \widetilde{g})$ is complete and has finite diameter, the
Hopf--Rinow theorem implies that $M$ is compact.

In particular, there is no complete noncompact finite-volume
positive global Yamabe minimizer. In the globally minimizing
case, the preceding argument gives an alternative proof of the
non-completeness phenomenon observed by
Gro{\ss}e~\cite{Gro} and is closely related to the compactness
results of Deng~\cite{D}. Since the complete finite-volume
setting was already studied in these earlier works, we retain
the closedness assumption in Theorem~\ref{thm:main} and record the
observation only as a remark.
\end{rmk}

\section{The diameter estimate for positive relative Yamabe metrics}

In this section, we review the definitions of relative Yamabe constants and relative Yamabe metrics (cf.\,\cite{AB1}).
Let $W^n$ be a connected, compact $n$-manifold ($n \geq 3$) with boundary $\partial W \ne \emptyset$
and $\mathcal{M}_0(W)$ the space of all {\it relative metrics} on $W$, that is, all Riemannian metrics $g$ on $W$ with zero mean curvature $H_g = 0$ on $\partial W$.
For $g \in \mathcal{M}_0(W)$, set
$$
[g]_0 := \{\widetilde{g} \in [g]\ |\ H_{\widetilde{g}} = 0\ {\rm on}\ \partial W \}
= \{ u^{\frac{4}{n-2}}g\ |\ u \in C^{\infty}_{> 0}(W)\ {\rm with}\ \partial_{\nu}u = 0\ {\rm along}\ \partial W \},
$$
which is called the {\it relative conformal class} of $g$.
Here, $\nu$ denotes the unit normal (inward) vector field along $\partial W$.
The restriction $E|_{[g]_0}$ of $E$ to any relative conformal class $[g]_0$ is also always bounded from below.
Hence, we can define the {\it relative Yamabe constant} of $[g]_0\ (g \in \mathcal{M}_0(W))$
$$
Y(W, \partial W; [g]_0) := \inf_{\widetilde{g} \in [g]_0}E(\widetilde{g}),
$$
which is also a conformal invariant.
A series of papers of Cherrier\,\cite{Cr}, Escobar\,\cite{E}, Brendle and Chen\,\cite{BC}, Almaraz, Barbosa and Lima\,\cite{ABL},
and Brendle and Wang\,\cite{BW} asserts that,
each relative conformal class $[g]_0$ on $W$ always contains a relative metric $\overset{\lor}{g} \in [g]_0$,
called a {\it relative Yamabe metrics},
which realizes the minimum
$$
Y(W, \partial W; [g]_0) = E(\overset{\lor}{g}).
$$
This metric must have constant scalar curvature
$$
R_{\overset{\lor}{g}} = Y(W, \partial W; [g]_0)\cdot V_{\overset{\lor}{g}}^{-2/n},
$$
where $V_{\overset{\lor}{g}} = {\rm Vol}(W, \overset{\lor}{g})$.
Cherrier proved that
$$
Y(W, \partial W; C_0) \leq Y(S^n_+, S^{n-1}; [g_{{}_{S^n_+}}]_0) = n(n-1) V_{g_{{}_{S^n_+}}}^{2/n}
$$
for any relative conformal class $C_0$ on $W$.
Here, $S^n_+$ denotes the $n$-hemisphere with the standard metric
$g_{{}_{S^n_+}} := g_{S^n}|_{S^n_+}$.

For $g \in \mathcal{M}_0(W)$, the relative Yamabe constant $Y(W, \partial W; [g]_0)$ is also rewritten as
\begin{align*}
Y(W, \partial W; [g]_0)
&= \inf_{\substack{u \in C^{\infty}_{> 0}(W)\\ \partial_{\nu}u = 0\ {\rm along}\ \partial W}}\frac{\alpha_n\int_W|du|^2d\mu_g + \int_WR_gu^2d\mu_g}
{\Big{(}\int_W u^{2n/(n-2)}d\mu_g\Big{)}^{(n-2)/n}}\ \ \cdots\ (*)\\
&= \inf_{\substack{u \in W^{1, 2}(W) \\ u \not\equiv 0}}\frac{\alpha_n\int_W|du|^2d\mu_g + \int_WR_gu^2d\mu_g}
{\Big{(}\int_W|u|^{2n/(n-2)}d\mu_g\Big{)}^{(n-2)/n}}\qquad \ \cdots\ (**).
\end{align*}
Note that, without the Neumann boundary condition : $\partial_{\nu}u = 0$ along $\partial W$,
the infimum of the non-linear Dirichlet quotient $(**)$ is equal to that of $(*)$.
Let $\overset{\lor}{g} \in [g]_0$ be a relative Yamabe metric.
If $Y(W, \partial W; [g]) > 0$, the above and $R_{\overset{\lor}{g}} = Y(W, \partial W; [g])\cdot V_{\overset{\lor}{g}}^{-2/n}$ imply
$$
\Big{(}\int_W|u|^{2n/(n-2)}d\mu_{\overset{\lor}{g}}\Big{)}^{(n-2)/n} \leq
\frac{\alpha_n}{Y(W, \partial W; [g])}\int_W|du|^2d\mu_{\overset{\lor}{g}} + \frac{1}{V_{\overset{\lor}{g}}^{2/n}}\int_Wu^2d\mu_{\overset{\lor}{g}}
$$
for any $u \in W^{1,2}(W)$.
Moreover, mlutiplying the above by $\frac{1}{V_{\overset{\lor}{g}}^{(n-2)/n}}$, we also get
\begin{equation}
\Big{(}\int_W|u|^{2n/(n-2)}\frac{1}{V_{\overset{\lor}{g}}}d\mu_{\overset{\lor}{g}}\Big{)}^{(n-2)/n} \leq
\frac{\alpha_n}{R_{\overset{\lor}{g}}}\int_W|du|^2\frac{1}{V_{\overset{\lor}{g}}}d\mu_{\overset{\lor}{g}}
+ \int_Wu^2\frac{1}{V_{\overset{\lor}{g}}}d\mu_{\overset{\lor}{g}}
\end{equation}
for any $u \in W^{1,2}(W)$.
Here, we notice that $\frac{1}{V_{\overset{\lor}{g}}}d\mu_{\overset{\lor}{g}}$ is a probability measure on $W$,
that is, $\frac{1}{V_{\overset{\lor}{g}}}\mu_{\overset{\lor}{g}}(W) = 1$.
The diameter ${\rm diam}(W, {\overset{\lor}{g}})$ of $(W, {\overset{\lor}{g}})$ is also rewritten as
\begin{align*}
{\rm diam}(W, {\overset{\lor}{g}})
&= \sup\{ ||\widetilde{f}||_{L^{\infty}}\ \big{|}\ f \in W^{1, 2}(W),\ ||df||_{L^{\infty}} \leq 1 \} \\
&= \sup\{ |\widetilde{f}|\ \big{|}\ f \in W^{1, 2}(W) \cap C^0(W),\ ||df||_{L^{\infty}} \leq 1 \},
\end{align*}
where $\widetilde{f}(x, y) := f(x) - f(y)$ for $x, y \in W$ (see the proof of Theolem\,3.1 below for detail).
Remark that, since $(W, g)$ is connected and compact, there always exists $f \in W^{1,2}(W)$
such that $||df||_{\infty} \leq 1$ and $||\widetilde{f}||_{\infty} = {\rm diam}(W, g)$.

Theorems\,2.1 and 2.2 still holds even if $\partial M$ is not empty.
We can obtain the diameter estimate below :

\begin{thm}
Let ${\overset{\lor}{g}}$ be a positive relative Yamabe metric on a connected, compact $n$-manifold $W$ with boundary\,$(n \geq 3)$.
Assume that $R_{\overset{\lor}{g}} \geq n(n - 1)k^2$\ $($ for some $k > 0$ $)$.
Then,
$$
{\rm diam}(W, {\overset{\lor}{g}}) \leq \frac{\pi}{k}.
$$
Moreover, if ${\rm diam}(W, {\overset{\lor}{g}}) = \frac{\pi}{k}$,
then $(W, {\overset{\lor}{g}})$ is isometric to the round $n$-hemisphere $S^n_+(\frac{1}{k})$.
\end{thm}

For the proof of the equality case in Theorem\,3.1, we also prepare the following result:

\begin{thm}$(${\rm Bakry--Ledoux}\,\cite[Theorems\,4]{BL}$).$
Let $(W, g)$ be a connected, compact Riemannian $n$-manifold with boundary $(n \geq 3)$ satisfying the Sobolev inequality $(4)$.
Assume that there exists a function $f \in W^{1,2}(W)$ such that $||df||_{\infty} \leq 1$ and $||\widetilde{f}||_{\infty} = \pi$.
Then, there exists a nonconstant function $u \in W^{1,2}(W)$ such that the equality holds in $(4)$ as below $:$
\begin{equation}
||u||^2_{2n/(n-2)} = A||du||^2_2 + ||u||^2_2.
\end{equation}
Moreover, if we set $X := \sin(f)$, for every $\varphi \in C^{\infty}(\mathbb{R})$,
$$
\Delta_g\varphi(X) = (1 - X^2)\varphi''(X) - nX\varphi'(X).
$$
\end{thm}


\begin{proof}[Proof of Theorem\,3.1]
Same as the proof of Theorem\,1.3, we can assume that $k = 1$.
Put $d\mu=\frac{1}{V_{\overset{\lor}{g}}}d\mu_{\overset{\lor}{g}}$. Consider the closed quadratic form
\[
 {\mathcal E}(v,w)=\int_W\langle dv,dw\rangle\,d\mu,
 \qquad
 D({\mathcal E})=W^{1,2}(W).
\]
Because its form domain is the full Sobolev space $W^{1,2}(W)$, the associated
self-adjoint Markov generator is the Neumann Laplacian. Its carr\'{e} du champ
is
\[
 \Gamma(v,v) = |dv|^2,
\]
and its intrinsic distance is
\[
 d_{\mathcal E}(x,y)
 :=\sup\{|v(x)-v(y)|\ \big{|}\ v\in W^{1,2}(W)\cap C^0(W),\ ||dv||_{L^{\infty}} \leq1\}
 ={\rm dist}_{\overset{\lor}{g}}(x,y).
\]
Indeed, the inequality $d_{\mathcal E}\leq {\rm dist}_{\overset{\lor}{g}}$ follows by
integrating $dv$ along curves, while the reverse inequality follows by using
the distance function from a fixed point. Thus the diameter in the abstract
Bakry--Ledoux theorem is exactly the Riemannian length-space diameter of
$(W, \overset{\lor}{g})$.

Applying Theorem\,2.1 to the Neumann form and to the Sobolev inequality\,$(2)$,
with $A=\alpha_n/R_{\overset{\lor}{g}}$, gives
\[
 {\rm diam}(W, \overset{\lor}{g})
 \leq\pi\sqrt{\frac{n(n-1)}{R_{\overset{\lor}{g}}}}
 \leq\pi.
\]
This proves the diameter estimate.

Assume now that ${\rm diam}(W, \overset{\lor}{g})=\pi$. Then equality in the preceding
chain forces
\[
 R_{\overset{\lor}{g}}=n(n-1),
 \qquad
 \frac{\alpha_n}{R_{\overset{\lor}{g}}}=\frac{4}{n(n-2)}.
\]
The equality statement in Theorem\,3.2, interpreted for the Neumann generator,
gives a nonconstant function $X\colon W\longrightarrow[-1,1]$ satisfying
\[
 \Delta_{\overset{\lor}{g}}\varphi(X)
 =(1-X^2)\varphi''(X)-nX\varphi'(X)
 \qquad\text{for every }\varphi\in C^\infty(\mathbb R).
\]
\underline{Claim\,2.}\quad $X \in C^{\infty}(W),\quad \partial_{\nu}X = 0$\ \ on \ $\partial W$.\\
\quad \\
{\it Proof of Claim\,{\rm 2}} :\ The Sobolev inequality\,(4) gives
\begin{equation}
\mathcal{F}(u) := \frac{4}{n(n-2)}\int_W|du|^2_{\overset{\lor}{g}}d\mu + \int_Wu^2d\mu
- \Big{(}\int_W|u|^{2n/(n-2)}d\mu\Big{)}^{(n-2)/n} \geq 0
\end{equation}
for any $u\in W^{1,2}(W)$.
In the Bakry-Ledoux framework\,\cite{BL}, we take $\mathcal A=\operatorname{Lip}(W)$.

Let $f\in\operatorname{Lip}(W)$ satisfy
$\|df\|_{L^\infty}\leq 1$ and
$\Vert \tilde f \Vert_\infty = \max_W f-\min_W f=\pi$.
Translate $f$ as in Theorem\, 3.2 so that if $X=\sin f$, then
\[
  \int_W X\,d\mu=0,\qquad -1\leq X\leq 1.
\]
Also by Theorem\,3.2,
\[
  u_\lambda=(1+\lambda X)^{- (n-2)/2},\qquad |\lambda|<1,
\]
satisfies $\mathcal F(u_\lambda)=0$.  Each $u_\lambda$ is positive and
Lipschitz, since $1+\lambda X\geq 1-|\lambda|>0$.

Fix $\psi\in C^\infty(W)$, smooth up to the boundary, without imposing
any boundary condition on $\psi$.  Since the inequality\,(6) holds on
all of $W^{1,2}(W)$, the variations $u_\lambda+t\psi$ are admissible.
Differentiating $\mathcal F(u_\lambda+t\psi)$ at $t=0$ and dividing by
$2$, therefore gives
\begin{equation}\label{eq:BLN-first-variation}
  \frac{4}{n(n-2)}\int_W\langle du_\lambda,d\psi\rangle_{\overset{\lor}{g}}\,d\mu
  +\int_W u_\lambda\psi\,d\mu
  =S_\lambda^{\,- 2/n}
    \int_W u_\lambda^{(n+2)/(n-2)}\psi\,d\mu,
 \end{equation}
where $S_\lambda :=\int_W u_\lambda^{2n/(n-2)}\,d\mu$.
Because $X$ is Lipschitz, the map $\lambda\mapsto u_\lambda$ is smooth
with values in $W^{1,\infty}(W)$ near $0$.  We may thus differentiate
\eqref{eq:BLN-first-variation} at $\lambda=0$.  At this parameter value,
\[
u_0=1,\qquad
  \left.\frac{d}{d\lambda}u_\lambda\right|_{\lambda=0}= - \frac{n-2}{2} X,
  \qquad S_0=1,\qquad
  S'_0 = - n\int_W X\,d\mu = 0.
\]

Hence we obtain
\[
  \frac{4}{n(n-2)}\int_W\langle dX,d\psi\rangle_{\overset{\lor}{g}}\,d\mu
  +\int_W X\psi\,d\mu
  = \frac{n+2}{n-2}\int_W X\psi\,d\mu.
\]
Since $C^\infty(W)$ is dense in $W^{1,2}(W)$, we have
\begin{equation}\label{eq:BLN-weak-eigenfunction}
  \int_W\langle dX,d\psi\rangle_{\overset{\lor}{g}}\,d\mu
  =n\int_W X\psi\,d\mu
  \qquad\text{for every }\psi\in W^{1,2}(W).
\end{equation}

Taking test functions supported in the interior first gives
$-\Delta_{\overset{\lor}{g}}X=nX$ in the interior $W^\circ$ of $W$.
But \eqref{eq:BLN-weak-eigenfunction} holds for test
functions with arbitrary boundary traces: it is the variational
formulation of the Neumann eigenvalue problem.
Hence boundary elliptic regularity gives $X\in W^{2,2}(W)$, and bootstrapping
the eigenvalue equation gives $X\in C^\infty(W)$, particularly smooth up to
$\partial W$.

To see that $X$ satisfies the Neumann condition, let $\nu$ denote the inward
unit normal and set
$d\sigma=\frac{1}{V_{\overset{\lor}{g}}}d\sigma_{\overset{\lor}{g}}$.
Green's formula then reads
\[
  \int_W\langle dX,d\psi\rangle_{\overset{\lor}{g}}\,d\mu
  =-\int_W\psi\,\Delta_{\overset{\lor}{g}}X\,d\mu
   -\int_{\partial W}\psi\,\partial_\nu X\,d\sigma.
\]
Comparing this with \eqref{eq:BLN-weak-eigenfunction} and using
$-\Delta_{\overset{\lor}{g}}X=nX$, we obtain
\[
  \int_{\partial W}\psi\,\partial_\nu X\,d\sigma=0
  \qquad\text{for every }\psi\in C^\infty(W).
\]
  Hence
\[
  \partial_\nu X=0\qquad\text{on }\partial W.
\]
This completes the proof of Claim\,2. \\

Taking $\varphi(x)=x$ and $\varphi(x)=x^2$ gives
\[
 \Delta_{\overset{\lor}{g}}X=-nX,
 \qquad
 |dX|^2=1-X^2.
\]
We next determine the critical points of $X$. At an interior extremum,
$dX=0$. At a boundary extremum, the tangential derivatives vanish and the
Neumann condition gives the normal derivative equal to zero, so again $dX=0$.
It follows from $|dX|^2=1-X^2$ that
\[
 \max_WX=1,
 \qquad
 \min_WX=-1.
\]
Every critical point lies in
\[
 P_+:=\{X=1\}
 \quad\text{or}\quad
 P_-:=\{X=-1\}.
\]
At a point of $P_+$, comparison of the quadratic terms in
$|dX|^2=1-X^2$, in an interior normal coordinate chart or in a boundary
half-coordinate chart, gives
\[
 ({}^{\overset{\lor}{g}}\nabla^2X)^2 = - {}^{\overset{\lor}{g}}\nabla^2X.
\]
Since ${\rm tr}_{\overset{\lor}{g}}({}^{\overset{\lor}{g}}\nabla^2X)=\Delta_{\overset{\lor}{g}}X=-n$ there,
\[
{}^{\overset{\lor}{g}}\nabla^2X=-I\quad\text{on }P_+.
\]
Similarly,
\[
 {}^{\overset{\lor}{g}}\nabla^2X=I\quad\text{on }P_-.
\]
Thus all critical points are nondegenerate and isolated.

The connectedness argument from Morse theory has a boundary version here.
On every compact slab $\{a\leq X\leq b\}\subset\{-1<X<1\}$, the vector field
\[
 Z=\frac{{}^{\overset{\lor}{g}}\nabla X}{|dX|^2}
\]
is smooth, satisfies $Z(X)=1$, and is tangent to $\partial W$ because
$\partial_\nu X=0$. Its flow therefore preserves the boundary and identifies
all regular sublevel sets between consecutive levels. For a level just above
$-1$, the sublevel set is a disjoint union of one ball for each interior
minimum and one half-ball for each boundary minimum. Since there are no other
critical points, these components cannot merge. On the other hand,
$W\setminus P_+$ is connected because $W$ is connected, $n\geq3$, and $P_+$
is finite. Hence $P_-$ consists of one point. Applying the same argument to
$-X$ shows that $P_+$ also consists of one point.

Both critical points lie on the boundary. Indeed, on every component of the
compact manifold $\partial W$, the restriction of $X$ attains a maximum and a
minimum. At such points the tangential gradient vanishes, and the Neumann
condition makes them critical points of $X$ on $W$. The restriction cannot be
constant, since otherwise an entire boundary component would consist of
critical points. Since $X$ has only one maximum and one minimum, these two
points must both belong to $\partial W$; in particular, $\partial W$ is
connected.

Set
\[
 t=\arccos X.
\]
On $W\setminus(P_+\cup P_-)$,
\[
 |dt|=1,
 \qquad
 \Delta_{\overset{\lor}{g}}t=(n-1)\cot t,
 \qquad
 \partial_\nu t=0\quad\text{on }\partial W.
\]
For $0<\tau<\pi$, define
\[
 \Sigma_\tau:=\{t=\tau\},
 \qquad
 A(\tau):={\rm Vol}_{n-1}(\Sigma_\tau).
\]
For $0<a<b<\pi$, the boundary of $\{a<t<b\}$ consists of $\Sigma_a$,
$\Sigma_b$, and the lateral part contained in $\partial W$. The flux through
the lateral part vanishes because $\partial_\nu t=0$. Consequently, the
divergence theorem and the coarea formula give
\begin{align*}
 A(b)-A(a)
 &=\int_{\{a<t<b\}}\Delta_{\overset{\lor}{g}}t\,d\mu_{\overset{\lor}{g}}\\
 &=\int_a^b(n-1)\cot\tau\,A(\tau)\,d\tau.
\end{align*}
Thus
\[
 A(\tau)=C\sin^{n-1}\tau
\]
for some $C>0$.

Let $p_+$ be the unique point where $X=1$. Since $p_+\in\partial W$ and
${}^{\overset{\lor}{g}}\nabla^2X(p_+)=-I$, a smooth boundary chart centered at $p_+$ identifies
$W$, to first order, with a Euclidean half-space and gives
\[
 t(y)=|y|+O(|y|^2).
\]
Therefore the level hypersurface $\Sigma_\tau$ is asymptotic to a Euclidean
hemisphere, and
\[
 A(\tau)=\frac{\sigma_{n-1}}2\tau^{n-1}
 +o(\tau^{n-1})
 \quad\text{as }\tau\downarrow0.
\]
It follows that $C=\sigma_{n-1}/2$. Hence
\[
 V_{\overset{\lor}{g}}
 =\int_0^\pi A(\tau)\,d\tau
 =\frac{\sigma_n}{2}
 =V_{g_{{}_{S^n_+}}}.
\]
Since $\overset{\lor}{g}$ is a relative Yamabe metric,
\[
 Y(W,\partial W;[g]_0)
 =R_{\overset{\lor}{g}}V_{\overset{\lor}{g}}^{2/n}
 =n(n-1)\left(\frac{\sigma_n}{2}\right)^{2/n}
 =Y(S^n_+,S^{n-1};[g_{{}_{S^n_+}}]_0).
\]

The rigidity in the equality case of the optimal relative Yamabe inequality
implies that $(W,\partial W;[g]_0)$ is conformally diffeomorphic to the
standard hemisphere (cf.\,\cite{Cr}, \cite{E}, \cite{BC}, \cite{ABL}, \cite{BW}). Let
$\Psi\colon W\longrightarrow S^n_+$ be such a conformal diffeomorphism and set
\[
 h=(\Psi^{-1})^*\overset{\lor}{g}.
\]
Then
\[
 h=v^{4/(n-2)}g_{{}_{S^n_+}}
\]
for a smooth positive function $v$, and
\[
 R_h=n(n-1),
 \qquad
 H_h=0.
\]
Since the equator is minimal for $g_{{}_{S^n_+}}$, the conformal
transformation law for mean curvature gives
\[
 \partial_\nu v=0\quad\text{on }S^{n-1}.
\]
The function $v$ satisfies the Yamabe equation on $S^n_+$ with this Neumann
condition. Its even reflection across the equator is therefore a weak solution
of the same Yamabe equation on $S^n$, and elliptic regularity makes the
reflection a smooth positive function $\overline{v}$. Thus
\[
 \overline{h}=\overline{v}^{4/(n-2)}g_{{}_{S^n}}
\]
is a smooth metric of scalar curvature $n(n-1)$ on $S^n$. By Obata's theorem\,\cite{Obata}, there exists a conformal diffeomorphism
$\phi\colon S^n\longrightarrow S^n$ such that
\[
 \overline{h} = \phi^*g_{{}_{S^n}}.
\]
The reflected metric $\overline{h}$ is invariant under reflection across the
equator. Hence $\phi$ carries the equator to the fixed hypersphere of an
involutive round isometry, and this hypersphere is a great sphere. After
postcomposing $\phi$ with a round isometry, we may therefore assume that
$\phi(S^n_+)=S^n_+$. Restricting to $S^n_+$ gives
\[
 h=\phi^*g_{{}_{S^n_+}}.
\]
Consequently,
\[
 \overset{\lor}{g}
 =\Psi^*h
 =(\phi\circ\Psi)^*g_{{}_{S^n_+}},
\]
so $\phi\circ\Psi$ is an isometry from $(W, \overset{\lor}{g})$ to the unit round
hemisphere. Undoing the initial scaling proves that, in the equality case,
$(W, \overset{\lor}{g})$ is isometric to $S^n_+(\frac{1}{k})$.
\end{proof}

\section{The distance estimate between the two conical points for positive almost conical Yamabe metrics}

Let $(Z^{n-1}, h)$ be a connected, closed Riemannian $(n-1)$-manifold ($n \geq 3$).
We then consider the cylinder $(Z \times \mathbb{R}, \overline{h} := h + dt^2)$.
One can define the Yamabe constant $Y(Z \times \mathbb{R}, [\overline{h}])$ of it, similarly to that of the closed case, by
\begin{align*}
Y(Z \times \mathbb{R}, [\overline{h}]) &= \inf_{\substack{f \in C^{\infty}_0(Z \times \mathbb{R}) \\ f \not\equiv 0}}
\frac{\alpha_n\int_{Z \times \mathbb{R}}|df|^2d\mu_{\overline{h}} + \int_{Z \times \mathbb{R}}R_{\overline{h}}f^2d\mu_{\overline{h}}}
{\Big{(}\int_{Z \times \mathbb{R}}|f|^{\frac{2n}{n-2}}d\mu_{\overline{h}}\Big{)}^{\frac{n-2}{n}}}\\
&= \inf_{\substack{f \in W^{1, 2}(Z \times \mathbb{R}) \\ f \not\equiv 0}}
\frac{\alpha_n\int_{Z \times \mathbb{R}}|df|^2d\mu_{\overline{h}} + \int_{Z \times \mathbb{R}}R_{\overline{h}}f^2d\mu_{\overline{h}}}
{\Big{(}\int_{Z \times \mathbb{R}}|f|^{\frac{2n}{n-2}}d\mu_{\overline{h}}\Big{)}^{\frac{n-2}{n}}}.
\end{align*}
Note that $- \infty \leq Y(Z \times \mathbb{R}, [\overline{h}]) \leq Y(S^n, [g_{{}_{S^n}}])
= Y(S^{n-1} \times \mathbb{R}, [\overline{g}_{{}_{S^{n-1}}} := g_{{}_{S^{n-1}}} + dt^2])$,
which may take $- \infty$.
Indeed, if $(Z, h)$ be a closed manifold of negative scalar curvature, then $Y(Z \times \mathbb{R}, [\overline{h}]) = - \infty$.
We define the operator (\cite{AB2})
$$
\mathcal{L}_h := -\,\alpha_n\Delta_h + R_h\qquad {\rm on}\quad (Z, h),
$$
which is different from the conformal Laplacian $L_h := -\,\alpha_{n-1}\Delta_h + R_h$ of $(Z, h)$.
Note that, if the first eigenvalue  $\lambda_1(\mathcal{L}_h)$ of $\mathcal{L}_h$ is positive, then  $Y(Z \times \mathbb{R}, [\overline{h}]) > 0$.
Hence, the positivity of $R_h$ also implies that $Y(Z \times \mathbb{R}, [\overline{h}]) > 0$.
Let $Q_{\bar{h}}$ denote the non-linear Dirichlet quotient associated to $Y(Z \times \mathbb{R}, [\overline{h}])$, that is,
$$
Q_{\overline{h}}(f) := \frac{\alpha_n\int_{Z \times \mathbb{R}}|df|^2d\mu_{\overline{h}} + \int_{Z \times \mathbb{R}}R_{\overline{h}}f^2d\mu_{\overline{h}}}
{\Big{(}\int_{Z \times \mathbb{R}}|f|^{\frac{2n}{n-2}}d\mu_{\overline{h}}\Big{)}^{\frac{n-2}{n}}}\qquad
{\rm for}\quad f \in W^{1, 2}(Z \times \mathbb{R})\ \ {\rm with}\ \ f \not\equiv 0.
$$
Let $\varphi_h \in C^{\infty}_{> 0}(Z)$ also denote the first eigenfunction of $\mathcal{L}_h$ with $\max_Z\varphi_h = 1$.

We now recall the following theorem (\cite[Theorems\,6.1 and 6.2]{AB2})\,:

\begin{thm}
Under these settings, we assume that $\lambda_1(\mathcal{L}_h) > 0$. \\
$({\rm i})$\ \ Assume that
$$
Y(Z \times \mathbb{R}, [\overline{h}]) < Y(S^n, [g_{{}_{S^n}}]) = Y(S^{n-1} \times \mathbb{R}, [\overline{g}_{{}_{S^{n-1}}}]).
$$
Then, there exists a positive function $u \in C^{\infty}_{> 0}(Z \times \mathbb{R})\,\cap\,W^{1,2}(Z \times \mathbb{R})$
with $\int_{Z \times \mathbb{R}}u^{\frac{2n}{n-2}}d\mu_{\overline{h}} = 1$ such that
$$
Q_{\overline{h}}(u) = Y(Z \times \mathbb{R}, [\overline{h}]),
$$
that is, $u$ is a Yamabe minimizer.
Moreover, for any constant $a > 0$ satisfying $0 < a < (\min_Z\varphi_h)^{2/(n-2)}\cdot a_0$,
there exist positive constants $\overline{C}, \underline{C} > 0$ and $\ell > 0$ such that
\begin{equation}
\underline{C}\cdot e^{-a_0|t|} \leq u(x, t) \leq \overline{C}\cdot e^{-a|t|}
\end{equation}
for $(x, t) \in Z \times [\ell, \infty)$, where $a_0 := \sqrt{\frac{\lambda_1(\mathcal{L})}{\alpha_n}} > 0$. \\
$({\rm ii})$\ \ Assume that
$$
Y(Z \times \mathbb{R}, [\overline{h}]) = Y(S^n, [g_{{}_{S^n}}]) = Y(S^{n-1} \times \mathbb{R}, [\overline{g}_{{}_{S^{n-1}}}]).
$$
Then, $(Z, h)$ is homothetic to the round $(n - 1)$-sphere $S^{n-1}(1) = (S^{n-1}, g_{{}_{S^{n-1}}})$.
\end{thm}
\begin{rmk}
In the above case $({\rm i})$,  since a Yamabe minimizer $u \in C^{\infty}_{> 0}(Z \times \mathbb{R})$ satisfies the inequalities\,$(9)$,
we shall call the Yamabe metric $\widehat{h} := u^{\frac{4}{n-2}}\cdot\overline{h}$ an {\it almost conical} metric.
In the above case $({\rm ii})$, $u(t) := (\cosh\,t)^{-\frac{n-2}{2}} \in C^{\infty}_{> 0}(S^{n-1} \times \mathbb{R})$
is a Yamabe minimizer and that the Yamabe metric $u^{\frac{4}{n-2}}\cdot\overline{g}_{{}_{S^{n-1}}} = (\cosh\,t)^{-2}\cdot\overline{g}_{{}_{S^{n-1}}}$
is isometric to the round metric $g_{{}_{S^n}}$ on $S^n(1) - \{(0, \cdots, 0, 1), (0, \cdots, 0, -1)\}$.
In both cases,
adding two points $p_{-}$ and $p_+$ corresponding to $Z \times \{- \infty\}$ and $Z \times \{+ \infty\}$ respectively,
we can compactify metrically $(Z \times \mathbb{R}, \widehat{h} = u^{\frac{4}{n-2}}\cdot\overline{h})$ to obtain a compact metric space
$(\widehat{Z \times \mathbb{R}}, \widehat{d}\,)$.
Here, $\{p_{-}, p_+\}$ are the two conical points of $(\widehat{Z \times \mathbb{R}}, \widehat{d}\,)$.
Note also that the condition $\int_{Z \times \mathbb{R}}u^{\frac{2n}{n-2}}d\mu_{\overline{h}} = 1$ implies $V_{\widehat{h}} = 1$
and $R_{\widehat{h}} = Y(Z \times \mathbb{R}, [\overline{h}])$.
\end{rmk}

If $Y(Z \times \mathbb{R}, [\overline{h}]) > 0$, then $V_{\widehat{h}} = 1, R_{\widehat{h}} = Y(Z \times \mathbb{R}, [\overline{h}])$ and the following
$$
Y(Z \times \mathbb{R}, [\overline{h}]) = \inf\,\{ Q_{\overline{h}}(f)\ |\ f \in C^{\infty}_0(Z \times \mathbb{R})\ \ {\rm with}\ \ f \not\equiv 0 \}
= \inf\,\{ Q_{\widehat{h}}(f)\ |\ f \in C^{\infty}_0(Z \times \mathbb{R})\ \ {\rm with}\ \ f \not\equiv 0 \}
$$
imply
$$
\Big{(}\int_{Z \times \mathbb{R}}|u|^{2n/(n-2)}d\mu_{\widehat{h}}\Big{)}^{(n-2)/n} \leq
\frac{\alpha_n}{R_{\widehat{h}}}\int_{Z \times \mathbb{R}}|du|^2d\mu_{\widehat{h}}
+ \int_{Z \times \mathbb{R}}u^2d\mu_{\widehat{h}}
$$
for any $u \in W^{1,2}(Z \times \mathbb{R}; \widehat{h})$.

Theorems\,2.1 and 2.2 still hold even if $(M, g)$ has conical singularities.
with this understood, we can obtain the following, which is motivated by Gromov's band width estimate\,\cite{G}.
\begin{thm}
Let $(Z \times \mathbb{R}, \overline{h} = h + dt^2)$ be a cylinder with $Y(Z \times \mathbb{R}, [\overline{h}]) > 0$.
Assume that its almost conical Yamabe metric $\widehat{h}$ satisfies $R_{\widehat{h}} \geq n(n - 1)k^2$\ $($ for some $k > 0$ $)$.
Then,
$$
\widehat{d}\,(p_{-}, p_+) \leq \frac{\pi}{k}.
$$
\end{thm}

\begin{proof}
Same as the proof of Theorem\,1.3, we can set $k = 1$. Then, by Theorem\,2.1, we obtain
$$
\widehat{d}\,(p_{-}, p_+) \leq {\rm diam}(Z \times \mathbb{R}, \widehat{h}) \leq \pi.
$$
\end{proof}

\begin{rmk} Remark that, under the condition $R_{\widehat{h}} \geq n(n - 1)$,
the equality $\widehat{d}\,(p_{-}, p_+) = \pi$ does not imply that
$(\widehat{Z \times \mathbb{R}}, \widehat{d}\,)$ is isometric to the round $n$-sphere $S^n(1)$.
Indeed, there exist many such examples of $(Z, h)$ below.
Let $\Gamma$ be a non-trivial spherical space form group acting freely on $S^{n-1}(1)$ by isometry.
Let $(Z, h) := S^{n-1}(1)/\Gamma$ be also the spherical space form, that is, the smooth quotient of $S^{n-1}(1)$ by $\Gamma$.
Then, $(\cosh\,t)^{-2}\cdot(h + dt^2)$ is an almost conical Yamabe metric
on $Z \times \mathbb{R}$ similar to Remark\,4.2,\,(ii) (cf.\,\cite[Section\,6]{AB2}).
This implies that
$$
\widehat{d}\,(p_{-}, p_+) = \pi,\qquad (\widehat{Z \times \mathbb{R}}, \widehat{d}\,) \ncong S^n(1).
$$
\end{rmk}

\vspace{10mm}

\bibliographystyle{amsbook}

\vspace{15mm}

\end{document}